\documentclass[onecolumn,10pt]{article}
\usepackage[top=.75in, bottom=.75in, left=.75in, right=.75in]{geometry}
\usepackage{amsmath,amsfonts,amscd,amssymb,amsthm,bbm,bm}
\usepackage{graphicx}
\usepackage{epstopdf}
\usepackage{overpic}
\usepackage{cancel}
\usepackage{rotating}
\usepackage{url}
\usepackage{caption}
\usepackage{color}
\usepackage{multirow}
\usepackage{wrapfig}
\usepackage{mathtools}
\usepackage{subeqnarray}
\usepackage{setspace}
\usepackage{palatino} 
\usepackage{microtype}
\usepackage[numbers,sort&compress]{natbib}

\usepackage[bottom,flushmargin,hang,multiple]{footmisc}
\usepackage{lipsum}
\newcommand\blfootnote[1]{%
\begingroup
\renewcommand\thefootnote{}\footnote{#1}%
\addtocounter{footnote}{-1}%
\endgroup
}

\definecolor{header1}{cmyk}{0,0,0,1}

\numberwithin{equation}{section}
\numberwithin{table}{section}
\numberwithin{equation}{section}
\theoremstyle{theorem}
\newtheorem{theorem}{Theorem}[section]
\newtheorem{proposition}[theorem]{Proposition}
\newtheorem{lemma}[theorem]{Lemma}
\newtheorem{corollary}[theorem]{Corollary}
\theoremstyle{definition}

\newtheorem{remark}[theorem]{Remark}
\newtheorem{example}[theorem]{Example}
\newtheorem{algorithm}[theorem]{Algorithm}

\usepackage[utf8]{inputenc}

\usepackage{hyperref}
\hypersetup{
colorlinks=false,      
pdfborder={0 0 0},     
citebordercolor={1 1 1}
}

\usepackage[normalem]{ulem}
\usepackage{color}

\title{\vspace{-.125in}{\huge\selectfont \textbf{The Waring Problem of Complex Binary Forms}}\vspace{-.075in}}

\author{\normalsize{Hua-Lin Huang$^{1*}$, Haoran Miao$^{2}$, and Yu Ye$^{3}$}\\
\footnotesize{$^{1}$ School of Mathematical Sciences, Huaqiao University, Quanzhou 362021, China} \\
\footnotesize{$^{2}$ School of Mathematical Sciences, Xiamen University, Xiamen 361005, China} \\
\footnotesize{$^{3}$ School of Mathematical Sciences, University of Science and Technology of China, Hefei 230026, China} \\
	\footnotesize{and} \\
	\footnotesize{Hefei National Laboratory, University of Science and Technology of China, Hefei 230088, China}
	}
	
	\date{}
	
	\providecommand{\subjclass}[2][2020]{%
\vspace{0.5em}%
\noindent\hspace{\leftmargini}\textbf{MSC 2020:} #2%
}

\providecommand{\keywords}[1]{%
\vspace{0.5em}%
\noindent\hspace{\leftmargini}\textbf{Keywords:} #1
}

\begin{document}
\maketitle



\blfootnote{$^*$ Corresponding author: hualin.huang@hqu.edu.cn}
\vspace{-.2in}
\begin{abstract}
	The Waring problem of forms concerns the expression of homogeneous multivariate polynomials as sums of powers of linear forms. This paper focuses on complex binary forms, and we solve the Waring problem for them using basic tools in algebra and analysis. In particular, we present elementary treatments of the Apolarity Lemma and Sylvester's 1851 Theorem, which are easily accessible and will provide an ideal approach for future extension to the general case. 
\end{abstract}



\section{Introduction}\label{sec:intro}
As one of the prestigious problems in mathematics, the classical Waring problem aims to represent all positive integers as sums of the least number of equal powers of integers \cite{Meditationes algebraicæ}. There is an analogy for homogeneous polynomials, or forms for brevity, which asks to express forms as sums of powers of linear forms. It is worth mentioning that, in Hilbert's landmark solution to the classical Waring problem, a Waring-like decomposition of powers of a sum of squares of variables plays a key role, see \cite{Hilbert, Ellison1}. 

In this paper, we focus on the Waring problem of complex binary forms. Given a binary form $f\in\mathbb{C}[x, y]_d$ of degree $d$, the Waring problem of $f$ is to find the least integer $r$ such that
\begin{equation*}
	f=\lambda_1L_1^d+\lambda_2L_2^d+\cdots+\lambda_rL_r^d, 
\end{equation*} 
where $\lambda_1,\lambda_2,\dots,\lambda_r\in\mathbb{C}^*$, $L_1,L_2,\dots,L_r\in\mathbb{C}[x, y]_1$ are linear forms. The above expression is referred to as a minimal decomposition of $f$, and $r$ is called the Waring rank of $f$, denoted by $\operatorname{WR}\,(f)=r$. 


The Waring problem of complex binary forms dates back to the investigations of canonical forms of polynomials by Sylvester \cite{Sylvester, S2}. He introduced the apolarity method and proved that a generic binary form of degree $2r-1$ has Waring rank $r$ and has the unique minimal decomposition. Gundelfinger extended the Hessian to higher orders and applied it to solve the Waring problem of complex binary forms explicitly \cite{Gundelfinger}. For nearly two centuries, the results and methods of Sylvester and Gundelfinger have been revisited and reformulated in increasingly modern language and with modern machinery, see e.g. \cite{Iarrobino,Kung Rota,Kung,even,Reznick}. Sylvester's work already contains the germ of an algorithm for the Waring decomposition of complex binary forms, which has since been refined and generalized in subsequent works \cite{BCMT,Comas,Comon}. Moreover, the Waring problem of multivariate forms has been systematically studied, see \cite{Iarrobino,AH}. 


This paper treats the Waring problem of complex binary forms via basic tools of linear algebra and calculus, keeping the exposition accessible to anyone with minimal mathematical preparation. By exploiting the familiar Vandermonde matrices to relate monomials in $\mathbb{C}[x, y]_d$ to the $d$-th powers of binary linear forms, we establish the existence of Waring decompositions and then apply polynomial differentiation to obtain an elementary proof of the Apolarity Lemma for binary forms. Furthermore, we provide a proof of the celebrated Sylvester's 1851 Theorem using linear equations and the theory of matrix eigenvalues. This treatment for the Waring problem of binary forms is not only elementary and transparent but also yields an explicit linear-algebraic interpretation of the summands: each term of the Waring decomposition corresponds to an eigenvalue of the Hankel matrix determined by the coefficients of the binary form. In addition, while treating the algorithm for Waring decomposition and Gundelfinger's Theorem for binary forms of even degree, we relate forms of different degrees via differentiation and integration. Our approach is natural, but seems novel. In contrast to Reznick's \cite{Reznick} ingenious elementary proof via generating functions, our approach relies solely on standard knowledge of algebra and analysis. 


The remainder of this paper is organized as follows. In Section \ref{Sec 2} we present an elementary proof of the Apolarity Lemma of binary forms and thereby establish the upper bound for their Waring ranks. In Section \ref{Sec 3} we use the coefficients of binary forms to present sufficient and necessary conditions for the Waring decomposition of binary forms, determine the Waring rank of a generic binary form, and provide an effective algorithm for computing both the Waring rank and a minimal decomposition of any given binary form.


Throughout this paper, we work over the polynomial ring $\mathcal{R}=\mathbb{C}[x,y]$ and denote by $\mathcal{R}_d$ the vector space of complex binary forms of degree $d$. By $\mathcal{D}=\mathbb{C}[\partial_x, \partial_y]$ we denote the ring of differential polynomials with complex coefficients. Any $\mathfrak{g} \in \mathcal{D}$ acts naturally on $f \in \mathcal{R}$ as a differential operator, written $\mathfrak{g} \circ f$. We set $\mathbb{C}^*=\mathbb{C}\setminus\{0\}$. For convenience, if $i \in \{1, 2, \dots, r \}$, we write $\check{i}=\{1,2,\dots,r\}\setminus\{i\}$. 


\section{The Apolarity Lemma}\label{Sec 2} 

The Apolarity Lemma is crucial for studying the Waring problem of forms. Using the Vandermonde matrix, we construct a basis of $\mathcal{R}_d$ consisting of the $d$-th powers of linear forms, establish the existence of Waring decompositions of binary forms, and leverage this result to prove the Apolarity Lemma.


\begin{lemma}\label{lemma 2.1}
	Let $\beta_1,\beta_2,\dots,\beta_r\in\mathbb{C}$ be pairwise distinct with $1\leqslant r\leqslant d+1$. Then the $r$ binary forms $(x+\beta_1y)^d,(x+\beta_2y)^d,\dots,(x+\beta_ry)^d$ are linearly independent. In particular, if $r=d+1$, then they constitute a basis of $\mathcal{R}_d$. 
\end{lemma}

\begin{proof}
	Note that the monomials $x^{d},\binom{d}{1}x^{d-1}y,\dots,\binom{d}{d-1}xy^{d-1},y^{d}$ form a basis of $\mathcal{R}_{d}$, and that
	\begin{equation*}
		\bigl((x+\beta_{1}y)^{d},(x+\beta_{2}y)^{d},\dots,(x+\beta_{r}y)^{d}\bigr)=\bigl(x^{d},\tbinom{d}{1}x^{d-1}y,\dots,y^{d}\bigr)
		\begin{pmatrix}
			1 & 1 & \cdots & 1\\
			\beta_{1} & \beta_{2} & \cdots & \beta_{r}\\
			\vdots & \vdots & \ddots & \vdots\\
			\beta_{1}^{d} & \beta_{2}^{d} & \cdots & \beta_{r}^{d}
		\end{pmatrix}.
	\end{equation*}
	By assumption $r\leqslant d+1$, so the above Vandermonde matrix has full column rank. Therefore, the $r$ binary forms $(x+\beta_1y)^d,(x+\beta_2y)^d,\dots,(x+\beta_ry)^d$ are linearly independent. The remaining follows trivially. 
\end{proof}

Lemma \ref{lemma 2.1} immediately yields the following result.

\begin{corollary}[Existence of Waring decompositions of binary forms]\label{Cor 2.2}
	For every $f\in \mathcal{R}_d$ and any choice of pairwise distinct $\beta_1,\beta_2,\dots,\beta_{d+1}\in\mathbb{C}$, there exist unique $\lambda_1,\lambda_2,\dots,\lambda_{d+1}\in\mathbb{C}$ such that
	\begin{equation*}
		f(x,y)=\sum_{k=1}^{d+1}\lambda_k(x+\beta_ky)^d. 
	\end{equation*}
\end{corollary}

\begin{remark}\label{Rmk 2.3}
	In fact, for any distinct $\beta_1,\beta_2,\dots,\beta_r\in\mathbb{C}$ with $1\leqslant r\leqslant d$, the $r+1$ binary forms $(x+\beta_1y)^d,(x+\beta_2y)^d,\dots,(x+\beta_ry)^d,y^d$ are linearly independent. In particular, setting $r=d$ yields a basis of $\mathcal{R}_d$. This approach, usually expressed in homogeneous coordinates, was previously employed by Reznick \cite{Reznick} and Ellison \cite{Ellison}, among others.
\end{remark}

Now we present an elementary proof of the Apolarity Lemma using Lemma \ref{lemma 2.1} and Corollary \ref{Cor 2.2}. 


\begin{theorem}[Apolarity Lemma]\label{Th 2.4}
	Let $f\in \mathcal{R}_d$ and let $\beta_1,\beta_2,\dots,\beta_r\in\mathbb{C}$ be pairwise distinct, where $1\leqslant r\leqslant d+1$. Then
	\begin{enumerate}
		\item[(1)] There exist $\lambda_1,\lambda_2,\dots,\lambda_r\in\mathbb{C}$ such that
		\begin{equation}\label{2.1}
			f(x,y)=\sum_{k=1}^{r}\lambda_k(x+\beta_ky)^d
		\end{equation}
		if and only if
		\begin{equation}\label{2.2}
			\prod_{\ell=1}^{r}(\partial_y-\beta_\ell\partial_x)\circ f=0.
		\end{equation}
		
		\item[(2)] For any fixed $1 \leqslant i \leqslant r$, the coefficient $\lambda_i$ in \eqref{2.1} is nonzero if and only if 
		\begin{equation*}
			\prod_{\ell\in\check{i}}(\partial_y-\beta_\ell\partial_x)\circ f\ne 0. 
		\end{equation*}
	\end{enumerate}
\end{theorem}

\begin{proof}
	(1) For any $1\leqslant k,\ell\leqslant r$, we have
	\begin{equation}\label{2.3}
		(\partial_y-\beta_\ell\partial_x)\circ(x+\beta_ky)^d=d(\beta_k-\beta_\ell)(x+\beta_ky)^{d-1}. 
	\end{equation}
	Assume that \eqref{2.1} holds. Then \eqref{2.3} implies
	\begin{align*}
		\prod_{\ell=1}^{r}(\partial_y-\beta_\ell\partial_x)\circ f&=\sum_{k=1}^{r}\lambda_k\biggl(\prod_{\ell=1}^{r}(\partial_y-\beta_\ell\partial_x)\circ(x+\beta_ky)^d\biggr)\\
		&=\sum_{k=1}^{r}\lambda_k\biggl(\prod_{\ell\in\check{k}}(\partial_y-\beta_\ell\partial_x)(\partial_y-\beta_k\partial_x)\circ(x+\beta_ky)^d\biggr)=0. 
	\end{align*}
	Conversely, suppose that \eqref{2.2} holds. Choose $\beta_{r+1},\dots,\beta_{d+1}\in\mathbb{C}$ so that $\beta_1,\dots,\beta_{d+1}$ are pairwise distinct. By Corollary \ref{Cor 2.2}, there exist $\lambda_1,\lambda_2,\dots,\lambda_{d+1}\in\mathbb{C}$ such that 
	\begin{equation}\label{d+1}
		f(x,y)=\sum_{k=1}^{d+1}\lambda_k(x+\beta_ky)^d. 
	\end{equation}
	Applying the operator $\prod_{\ell=1}^{r}(\partial_y-\beta_\ell\partial_x)$ to \eqref{d+1} yields
	\begin{equation*}
		\prod_{\ell=1}^{r}(\partial_y-\beta_\ell\partial_x)\circ f=\sum_{k=r+1}^{d+1}\lambda_k\frac{d!}{(d-r)!}\prod_{\ell=1}^{r}(\beta_k-\beta_\ell)(x+\beta_ky)^{d-r}=0, 
	\end{equation*}
	where the first $r$ terms vanish by \eqref{2.3}. According to Lemma \ref{lemma 2.1}, the binary forms $(x+\beta_{r+1}y)^{d-r},\dots,(x+\beta_{d+1}y)^{d-r}$ are linearly independent. Hence $\lambda_{r+1}=\cdots=\lambda_{d+1}=0$, and \eqref{2.1} follows. 
	
	(2) Assume that \eqref{2.1} holds. Notice that
	\begin{align*}
		\prod_{\ell\in\check{i}}(\partial_y-\beta_\ell\partial_x)\circ f&=\prod_{\ell\in\check{i}}(\partial_y-\beta_\ell\partial_x)\circ \biggl(\sum_{k\in\check{i}}\lambda_k(x+\beta_ky)^d+\lambda_i(x+\beta_iy)^d\biggr)\\
		&=\prod_{\ell\in\check{i}}(\partial_y-\beta_\ell\partial_x)\circ \lambda_i(x+\beta_iy)^d\\
		&=\lambda_i\frac{d!}{(d-r+1)!}\prod_{\ell\in\check{i}}(\beta_i-\beta_\ell)(x+\beta_iy)^{d-r+1}. 
	\end{align*}
	Thus, the conclusion follows immediately. 
\end{proof}

\begin{remark}
	Previous proofs of the Apolarity Lemma rely on algebraic geometry or invariant theory \cite{Iarrobino, Kung Rota}, which can be rather abstract and challenging for beginners to grasp. We complete the proof using only elementary linear algebra and calculus, and explicitly identify the conditions under which each term in a Waring decomposition actually appears. This may provide an ideal approach for further study of the general case.
	
\end{remark}

Furthermore, we derive the following result.

\begin{corollary}\label{Cor 2.6}
	The Waring rank of a binary form $f\in\mathcal{R}_d$ is the minimal degree of all homogeneous polynomials in $f^{\bot}=\{\mathfrak{g}\in\mathbb{C}[\partial_x,\partial_y]\mid \mathfrak{g}\circ f=0\}$ which have no repeated roots. 
\end{corollary}

We now review the well-known result that the Waring rank of a complex binary form of degree $d$ is at most $d$, which is a typical application of the Apolarity Lemma.

\begin{corollary}\label{Cor 2.7}
	The Waring rank of any $f\in \mathcal{R}_d$ is at most $d$. 
\end{corollary}

\begin{proof}
	Suppose $f(x,y)=a_0x^d+\binom{d}{1}a_1x^{d-1}y+\cdots+\binom{d}{d-1}a_{d-1}xy^{d-1}+a_dy^d$. 
	
	If $a_0a_d\ne 0$, perform the following linear changes of variables:
	\begin{equation*}
		\begin{cases}
			\tilde{x}=\sqrt[d]{a_0}x,\\
			\tilde{y}=\sqrt[d]{a_d}y.
		\end{cases}
	\end{equation*}
	Then the coefficients of  $\tilde{x}^d$ and $\tilde{y}^d$ are both $1$. 
	
	If exactly one of $a_0$ and $a_d$ is nonzero, assume without loss of generality that $a_0\ne 0$ and $a_d=0$. Then the polynomial $a_0x^d+\binom{d}{1}a_1x^{d-1}+\cdots+\binom{d}{d-1}a_{d-1}x$ is not identically zero and thus has at most $d$ roots. Choose $\alpha$ such that $a_0\alpha^d+\binom{d}{1}a_1\alpha^{d-1}+\cdots+\binom{d}{d-1}a_{d-1}\alpha \ne 0$, and perform the following linear changes of variables:
	\begin{equation*}
		\begin{cases}
			\tilde{x}=x-\alpha y,\\
			\tilde{y}=y.
		\end{cases}
	\end{equation*}
	Then the coefficient of $\tilde{y}^d$ is nonzero, while the coefficient of $\tilde{x}^d$ remains $a_0$. This reduces the problem to the case where $a_0a_d\ne 0$. 
	
	Hence, without loss of generality, we may assume that either $a_0=a_d=0$ or $a_0=a_d=1$. Let $\mathfrak{g}(\partial_x,\partial_y)=\partial_y^d-\partial_x^d=\prod_{k=1}^{d}(\partial_y-\zeta^k\partial_x)\in\mathbb{C}[\partial_x,\partial_y]$, where $\zeta=e^{\frac{2\pi \operatorname{i}}{d}}$. Then $\mathfrak{g}\circ f=0$. Thus, by Theorem \ref{Th 2.4}, there exist $\lambda_1,\lambda_2,\dots,\lambda_d\in\mathbb{C}$ such that
	\begin{equation*}
		f(x,y)=\sum_{k=1}^{d}\lambda_k(x+\zeta^ky)^d. 
	\end{equation*}
	
	In summary, the Waring rank of a complex binary form of degree $d$ is at most $d$.
\end{proof}

The following example demonstrates that there exists a binary form of degree $d$ whose Waring rank is exactly $d$.

\begin{example}
	Consider the form $f(x,y)=xy^{d-1}$. Let $f^{\bot}=\{\mathfrak{g}\in\mathbb{C}[\partial_x,\partial_y]\mid \mathfrak{g}\circ f=0\}$. For $i,j\geqslant 0$, we have 
	\begin{equation}\label{2.4}
		\partial_x^i\partial_y^j\circ f=
		\begin{cases}
			0&\operatorname{if}\ i\geqslant2\ \operatorname{or}\ j\geqslant d,\\
			\frac{(d-1)!}{(d-1-j)!}x^{1-i}y^{d-1-j}&\operatorname{otherwise}.
		\end{cases}
	\end{equation}
	It is clear that any $\mathfrak{g}\in\mathbb{C}[\partial_x,\partial_y]$ can be written in the following form: 
	\begin{equation*}
		\mathfrak{g}=\partial_x^2\cdot \mathfrak{g}_1+\partial_y^d\cdot \mathfrak{g}_2+\sum_{\substack{0\leqslant i\leqslant 1\\0\leqslant j\leqslant d-1}}a_{ij}\partial_x^i\partial_y^j, 
	\end{equation*}
	where $\mathfrak{g}_1,\mathfrak{g}_2\in\mathbb{C}[\partial_x,\partial_y]$ and $a_{ij}\in\mathbb{C}$. If $\mathfrak{g}\in f^{\bot}$, then by \eqref{2.4}, we have
	\begin{equation*}
		\mathfrak{g}\circ f=0+0+\sum_{\substack{0\leqslant i\leqslant 1\\0\leqslant j\leqslant d-1}}a_{ij}\frac{(d-1)!}{(d-1-j)!}x^{1-i}y^{d-1-j}=0. 
	\end{equation*}
	Since $1,y,\dots,y^{d-1},x,xy,\dots,xy^{d-1}$ are linearly independent, it follows that $a_{ij}=0$ for $0\leqslant i\leqslant 1$ and $0\leqslant j\leqslant d-1$. Therefore, $f^{\bot}=\langle\partial_x^2,\partial_y^d\rangle$. 
	
	Now it follows that, if $\mathfrak{g}\in f^\bot$ and $\operatorname{deg}\,\mathfrak{g}<d$, then $\partial_x^2\mid \mathfrak{g}$, hence $\mathfrak{g}$ has a repeated root. Combining Corollary \ref{Cor 2.6} and Corollary \ref{Cor 2.7}, we see that the Waring rank of $xy^{d-1}$ is exactly $d$. 
\end{example}

We conclude this section by presenting a specific example.

\begin{example}\label{Ex 2.9}
	Consider the binary form $f(x,y)=3x^3-3x^2y+9xy^2-y^3$. By direct calculation, we have
	\begin{gather*}
		\partial_x\circ f=9x^2-6xy+9y^2,\quad \partial_y\circ f=-3x^2+18xy-3y^2,\\
		\partial_x^2\circ f=18x-6y,\quad \partial_x\partial_y\circ f=-6x+18y,\quad \partial_y^2\circ f=18x-6y.
	\end{gather*}
	A simple observation shows that $\partial_x\circ f$ and $\partial_y\circ f$ are linearly independent. Thus, there does not exist a linear polynomial $\mathfrak{g}\in\mathbb{C}[\partial_x,\partial_y]$ such that $\mathfrak{g}\circ f=0$. Moreover, notice that 
	\begin{equation*}
		0=(\partial_y^2-\partial_x^2)\circ f=(\partial_y-\partial_x)(\partial_y+\partial_x)\circ f. 
	\end{equation*}
	Thus, by Corollary \ref{Cor 2.6}, the Waring rank of $f$ is $2$. 
	
	According to Theorem \ref{Th 2.4}, there exist $\lambda_1,\lambda_2\in\mathbb{C}$ such that $f(x,y)=\lambda_1(x+y)^3+\lambda_2(x-y)^3$. By the method of undetermined coefficients, we find $\lambda_1=1$ and $\lambda_2=2$. Hence, we obtain the following minimal decomposition of $f$: 
	\begin{equation*}
		f(x,y)=(x+y)^3+2(x-y)^3. 
	\end{equation*}
\end{example}

\section{The Waring decompositions}\label{Sec 3} 

In this section, we apply the crucial tools developed in Section \ref{Sec 2} and utilize the Hankel matrices constructed from coefficients of complex binary forms to address their Waring decompositions. Consequently, we determine the Waring ranks of binary forms and provide an algorithm for computing their minimal decompositions. 


\begin{theorem}\label{Th 3.1}
	Consider the complex binary form of degree $d$
	\begin{equation*}
		f(x,y)=\sum_{i=0}^{d}\binom{d}{i}a_ix^{d-i}y^i. 
	\end{equation*}
	Let $\beta_1,\beta_2,\dots,\beta_r\in\mathbb{C}$ be pairwise distinct, where $1\leqslant r\leqslant d$. Then the following statements are equivalent: 
	\begin{enumerate}
		\item[(1)] There exist $\lambda_1,\lambda_2,\dots,\lambda_r\in\mathbb{C}^*$ such that 
		\begin{equation*}
			f(x,y)=\sum_{k=1}^{r}\lambda_k(x+\beta_ky)^d. 
		\end{equation*}
		
		\item[(2)] 
		\begin{enumerate}
			\item There exist $c_0,c_1,\dots,c_{r-1}\in\mathbb{C}$ such that 
			\begin{equation}\label{3.1}
				a_i=-(c_0a_{i-r}+c_1a_{i-r+1}+\cdots+c_{r-1}a_{i-1}),\quad i=r,r+1,\dots,d. 
			\end{equation}
			For $i>d$, define inductively
			\begin{equation}\label{3.2}
				a_i=-(c_0a_{i-r}+c_1a_{i-r+1}+\cdots+c_{r-1}a_{i-1}). 
			\end{equation}
			
			\item 
			\begin{equation*}
				\operatorname{rank}\,\begin{pmatrix}
					a_0&a_1&\cdots&a_{r-1}\\
					a_1&a_2&\cdots&a_r\\
					\vdots&\vdots&\ddots&\vdots\\
					a_{r-1}&a_r&\cdots&a_{2r-2}
				\end{pmatrix}=r. 
			\end{equation*}
			
			\item The eigenvalues of the matrix
			\begin{equation*}
				\begin{pmatrix}
					a_0&a_1&\cdots&a_{r-1}\\
					a_1&a_2&\cdots&a_r\\
					\vdots&\vdots&\ddots&\vdots\\
					a_{r-1}&a_r&\cdots&a_{2r-2}
				\end{pmatrix}^{-1}
				\begin{pmatrix}
					a_1&a_2&\cdots&a_r\\
					a_2&a_3&\cdots&a_{r+1}\\
					\vdots&\vdots&\ddots&\vdots\\
					a_r&a_{r+1}&\cdots&a_{2r-1}
				\end{pmatrix}
			\end{equation*}
			are $\beta_1,\beta_2,\dots,\beta_r$. 
		\end{enumerate}
	\end{enumerate}
\end{theorem}

\begin{proof}
	For convenience of the exposition, we denote $\Lambda=\operatorname{diag}\,(\lambda_1,\lambda_2,\dots,\lambda_r)$, $D=\operatorname{diag}\,(\beta_1,\beta_2,\dots,\beta_r)$, and
	\begin{equation*}
		A_0=\begin{pmatrix}
			a_0&a_1&\cdots&a_{r-1}\\
			a_1&a_2&\cdots&a_r\\
			\vdots&\vdots&\ddots&\vdots\\
			a_{r-1}&a_r&\cdots&a_{2r-2}
		\end{pmatrix},\quad
		A_1=\begin{pmatrix}
			a_1&a_2&\cdots&a_r\\
			a_2&a_3&\cdots&a_{r+1}\\
			\vdots&\vdots&\ddots&\vdots\\
			a_r&a_{r+1}&\cdots&a_{2r-1}
		\end{pmatrix},\quad
		B=\begin{pmatrix}
			1&\beta_1&\cdots&\beta_1^{r-1}\\
			1&\beta_2&\cdots&\beta_2^{r-1}\\
			\vdots&\vdots&\ddots&\vdots\\
			1&\beta_r&\cdots&\beta_r^{r-1}
		\end{pmatrix}. 
	\end{equation*}
	
	(1)$\Rightarrow$(2). Suppose that
	\begin{equation*}
		f(x,y)=\sum_{k=1}^{r}\lambda_k(x+\beta_ky)^d.
	\end{equation*}
	Then we have
	\begin{equation*}
		f(x,y)=\sum_{i=0}^{d}\binom{d}{i}\biggl(\sum_{k=1}^{r}\lambda_k\beta_k^i\biggl)x^{d-i}y^i. 
	\end{equation*}
	By comparing coefficients, we obtain
	\begin{equation*}
		a_i=\lambda_1\beta_1^i+\lambda_2\beta_2^i+\cdots+\lambda_r\beta_r^i,\quad i=0,1,\dots,d. 
	\end{equation*}
	For all $i>d$, we define
	\begin{equation*}
		a_i=\lambda_1\beta_1^i+\lambda_2\beta_2^i+\cdots+\lambda_r\beta_r^i. 
	\end{equation*}
	
	Since $B$ is a Vandermonde matrix in the distinct $\beta_1,\beta_2,\dots,\beta_r\in\mathbb{C}$, it is invertible. Therefore, there exists a unique set $(-c_0,-c_1,\dots,-c_{r-1})\in\mathbb{C}^r$ such that 
	\begin{equation*}
		\begin{pmatrix}
			\beta_1^r\\
			\beta_2^r\\
			\vdots\\
			\beta_r^r
		\end{pmatrix}=
		\begin{pmatrix}
			1&\beta_1&\cdots&\beta_1^{r-1}\\
			1&\beta_2&\cdots&\beta_2^{r-1}\\
			\vdots&\vdots&\ddots&\vdots\\
			1&\beta_r&\cdots&\beta_r^{r-1}
		\end{pmatrix}
		\begin{pmatrix}
			-c_0\\
			-c_1\\
			\vdots\\
			-c_{r-1}
		\end{pmatrix}. 
	\end{equation*}
	Thus, for all $i\geqslant r$, we have 
	\begin{equation*}
		\begin{pmatrix}
			\beta_1^i\\
			\beta_2^i\\
			\vdots\\
			\beta_r^i
		\end{pmatrix}=
		\begin{pmatrix}
			\beta_1^{i-r}&\beta_1^{i-r+1}&\cdots&\beta_1^{i-1}\\
			\beta_2^{i-r}&\beta_2^{i-r+1}&\cdots&\beta_2^{i-1}\\
			\vdots&\vdots&\ddots&\vdots\\
			\beta_r^{i-r}&\beta_r^{i-r+1}&\cdots&\beta_r^{i-1}
		\end{pmatrix}
		\begin{pmatrix}
			-c_0\\
			-c_1\\
			\vdots\\
			-c_{r-1}
		\end{pmatrix}. 
	\end{equation*}
	Therefore, for all $i\geqslant r$, we have 
	\begin{align*}
		a_i&=(\lambda_1,\lambda_2,\dots,\lambda_r)
		\begin{pmatrix}
			\beta_1^i\\
			\beta_2^i\\
			\vdots\\
			\beta_r^i
		\end{pmatrix}=(\lambda_1,\lambda_2,\dots,\lambda_r)
		\begin{pmatrix}
			\beta_1^{i-r}&\beta_1^{i-r+1}&\cdots&\beta_1^{i-1}\\
			\beta_2^{i-r}&\beta_2^{i-r+1}&\cdots&\beta_2^{i-1}\\
			\vdots&\vdots&\ddots&\vdots\\
			\beta_r^{i-r}&\beta_r^{i-r+1}&\cdots&\beta_r^{i-1}
		\end{pmatrix}
		\begin{pmatrix}
			-c_0\\
			-c_1\\
			\vdots\\
			-c_{r-1}
		\end{pmatrix}\nonumber\\
		&=(a_{i-r},a_{i-r+1},\dots,a_{i-1})
		\begin{pmatrix}
			-c_0\\
			-c_1\\
			\vdots\\
			-c_{r-1}
		\end{pmatrix}=-(c_0a_{i-r}+c_1a_{i-r+1}+\cdots+c_{r-1}a_{i-1}). 
	\end{align*}
	Hence item (a) follows. 
	
	From the equation
	\begin{equation*}
		a_{i+j-2}=(\beta_1^{i-1},\beta_2^{i-1},\dots,\beta_r^{i-1})
		\begin{pmatrix}
			\lambda_1&&&\\
			&\lambda_2&&\\
			&&\ddots&\\
			&&&\lambda_r
		\end{pmatrix}
		\begin{pmatrix}
			\beta_1^{j-1}\\
			\beta_2^{j-1}\\
			\vdots\\
			\beta_r^{j-1}
		\end{pmatrix}, 
	\end{equation*}
	we see that $A_0=B^{\operatorname{T}} \Lambda B$. Similarly, we have $A_1=B^{\operatorname{T}} \Lambda DB$. 
	
	Since $\beta_1,\beta_2,\dots,\beta_r$ are pairwise distinct and $\lambda_k\ne 0$ for $k=1,2,\dots,r$, both $B$ and $\Lambda$ are invertible. Consequently, $A_0$ is invertible, i.e., $\operatorname{rank}\,(A_0)=r$. Moreover, 
	\begin{equation*}
		A_0^{-1}A_1=(B^{\operatorname{T}} \Lambda B)^{-1}B^{\operatorname{T}} \Lambda DB=B^{-1}DB, 
	\end{equation*}
	which implies that the eigenvalues of $A_0^{-1}A_1$ are $\beta_1,\beta_2,\dots,\beta_r$. Thus, items (b) and (c) are proved. 
	
	(2)$\Rightarrow$(1). Assume that items (a), (b) and (c) hold. From \eqref{3.1} and \eqref{3.2}, we have the matrix equation 
	\begin{equation*}
		\begin{pmatrix}
			a_0&a_1&\cdots&a_{r-1}\\
			a_1&a_2&\cdots&a_r\\
			\vdots&\vdots&\ddots&\vdots\\
			a_{r-1}&a_r&\cdots&a_{2r-2}
		\end{pmatrix}
		\begin{pmatrix}
			-c_0\\
			-c_1\\
			\vdots\\
			-c_{r-1}
		\end{pmatrix}=
		\begin{pmatrix}
			a_r\\
			a_{r+1}\\
			\vdots\\
			a_{2r-1}
		\end{pmatrix}. 
	\end{equation*}
	Thus, 
	\begin{equation*}
		\begin{pmatrix}
			a_1&a_2&\cdots&a_r\\
			a_2&a_3&\cdots&a_{r+1}\\
			\vdots&\vdots&\ddots&\vdots\\
			a_r&a_{r+1}&\cdots&a_{2r-1}
		\end{pmatrix}=
		\begin{pmatrix}
			a_0&a_1&\cdots&a_{r-1}\\
			a_1&a_2&\cdots&a_r\\
			\vdots&\vdots&\ddots&\vdots\\
			a_{r-1}&a_r&\cdots&a_{2r-2}
		\end{pmatrix}
		\begin{pmatrix}
			0&\cdots&0&-c_0\\
			1&\cdots&0&-c_1\\
			\vdots&\ddots&\vdots&\vdots\\
			0&\cdots&1&-c_{r-1}
		\end{pmatrix}. 
	\end{equation*}
	From (b), we see that $A_0$ is invertible. Hence, 
	\begin{equation*}
		\begin{pmatrix}
			a_0&a_1&\cdots&a_{r-1}\\
			a_1&a_2&\cdots&a_r\\
			\vdots&\vdots&\ddots&\vdots\\
			a_{r-1}&a_r&\cdots&a_{2r-2}
		\end{pmatrix}^{-1}
		\begin{pmatrix}
			a_1&a_2&\cdots&a_r\\
			a_2&a_3&\cdots&a_{r+1}\\
			\vdots&\vdots&\ddots&\vdots\\
			a_r&a_{r+1}&\cdots&a_{2r-1}
		\end{pmatrix}=
		\begin{pmatrix}
			0&\cdots&0&-c_0\\
			1&\cdots&0&-c_1\\
			\vdots&\ddots&\vdots&\vdots\\
			0&\cdots&1&-c_{r-1}
		\end{pmatrix}. 
	\end{equation*}
	Therefore, the characteristic polynomial of $A_0^{-1}A_1$ is 
	\begin{equation}\label{eigenpoly}
		T(\lambda)=\lambda^r+c_{r-1}\lambda^{r-1}+\cdots+c_1\lambda+c_0. 
	\end{equation}
	Since $\beta_1,\beta_2,\dots,\beta_r$ are the eigenvalues of $A_0^{-1}A_1$ by (c), we have $T(\beta_k)=0$ for $k=1,2,\dots,r$, that is, 
	\begin{equation*}
		\begin{pmatrix}
			\beta_1^r\\
			\beta_2^r\\
			\vdots\\
			\beta_r^r
		\end{pmatrix}+c_{r-1}
		\begin{pmatrix}
			\beta_1^{r-1}\\
			\beta_2^{r-1}\\
			\vdots\\
			\beta_r^{r-1}
		\end{pmatrix}+\cdots+c_1
		\begin{pmatrix}
			\beta_1\\
			\beta_2\\
			\vdots\\
			\beta_r
		\end{pmatrix}+c_0
		\begin{pmatrix}
			1\\
			1\\
			\vdots\\
			1
		\end{pmatrix}=
		\begin{pmatrix}
			0\\
			0\\
			\vdots\\
			0
		\end{pmatrix}. 
	\end{equation*}
	Let $\mathbf{b}_i=(\beta_1^i,\beta_2^i,\dots,\beta_r^i)^{\operatorname{T}}$ for $i=0,1,\dots,r$. Then $\mathbf{b}_r=-(c_{r-1}\mathbf{b}_{r-1}+\cdots+c_1\mathbf{b}_1+c_0\mathbf{b}_0)$ and $B=(\mathbf{b}_0,\mathbf{b}_1,\dots,\mathbf{b}_{r-1})$. Therefore, 
	\begin{equation}\label{3.3}
		BA_0^{-1}A_1=(\mathbf{b}_0,\mathbf{b}_1,\dots,\mathbf{b}_{r-1})
		\begin{pmatrix}
			0&\cdots&0&-c_0\\
			1&\cdots&0&-c_1\\
			\vdots&\ddots&\vdots&\vdots\\
			0&\cdots&1&-c_{r-1}
		\end{pmatrix}=(\mathbf{b}_1,\mathbf{b}_2,\dots,\mathbf{b}_r)=DB. 
	\end{equation}
	Since $B$ is invertible, there exists a unique vector  $(\lambda_1,\lambda_2,\dots,\lambda_r)\in\mathbb{C}^r$ such that $(a_0,a_1,\dots,a_{r-1})=(\lambda_1,\lambda_2,\dots,\lambda_r)B$. Thus, 
	\begin{align*}
		(a_1,a_2,\dots,a_r)&=(a_0,a_1,\dots,a_{r-1})A_0^{-1}A_1\\
		&=(\lambda_1,\lambda_2,\dots,\lambda_r)BA_0^{-1}A_1\\
		&=(\lambda_1,\lambda_2,\dots,\lambda_r)DB. 
	\end{align*}
	According to \eqref{3.1}, \eqref{3.2}, \eqref{3.3} and by induction, for any $i\geqslant 1$, the following holds: 
	\begin{align*}
		(a_i,a_{i+1},\dots,a_{i+r-1})&=(a_{i-1},a_i,\dots,a_{i+r-2})A_0^{-1}A_1\\
		&=\cdots\\
		&=(a_0,a_1,\dots,a_{r-1})(A_0^{-1}A_1)^i\\
		&=(\lambda_1,\lambda_2,\dots,\lambda_r)DB(A_0^{-1}A_1)^{i-1}\\
		&=\cdots\\
		&=(\lambda_1,\lambda_2,\dots,\lambda_r)D^iB. 
	\end{align*}
	Hence, for any $i\geqslant 0$, we have
	\begin{equation}\label{3.5}
		a_i=\lambda_1\beta_1^i+\lambda_2\beta_2^i+\cdots+\lambda_r\beta_r^i. 
	\end{equation}
	Consequently, 
	\begin{equation*}
		f(x,y)=\sum_{i=0}^{d}\binom{d}{i}\biggl(\sum_{k=1}^{r}\lambda_k\beta_k^i\biggl)x^{d-i}y^i=\sum_{k=1}^{r}\lambda_k(x+\beta_ky)^d. 
	\end{equation*}
	
	From \eqref{3.5}, we have 
	\begin{equation*}
		a_{i+j-2}=(\beta_1^{i-1},\beta_2^{i-1},\dots,\beta_r^{i-1})
		\begin{pmatrix}
			\lambda_1&&&\\
			&\lambda_2&&\\
			&&\ddots&\\
			&&&\lambda_r
		\end{pmatrix}
		\begin{pmatrix}
			\beta_1^{j-1}\\
			\beta_2^{j-1}\\
			\vdots\\
			\beta_r^{j-1}
		\end{pmatrix}. 
	\end{equation*}
	Thus, we have $A_0=B^{\operatorname{T}} \Lambda B$. Given that $A_0$ is invertible, it follows that $\lambda_k\ne 0$ for $k=1,2,\dots,r$. 
\end{proof}

\begin{remark}
	The above theorem demonstrates that each term $(x+\beta_k y)^d$ in a Waring decomposition of a given binary form corresponds to the eigenvalue $\beta_k$ of $A_0^{-1}A_1$, thereby providing a linear-algebraic interpretation for the decomposition summands. Moreover, the coefficients of the terms in a Waring decomposition also exhibit a clear linear-algebraic significance. They are the diagonal elements of the matrix obtained via the congruence of the Hankel matrix $A_0$ by the Vandermonde matrix determined by the aforementioned eigenvalues.
	Somehow this explains why the Waring decomposition of forms is also referred to as an interpolation problem of multivariate polynomials \cite{AH}. On the other hand, if a Waring decomposition of a binary form includes the term $y^d$, this situation corresponds to the case where the eigenvalue $\beta$ is $\infty$ in the theorem, thus necessitating separate treatment. 
	
\end{remark}

Suppose $f\in\mathcal{R}_d$ has a Waring decomposition of the form given in Theorem \ref{Th 3.1}:
\begin{equation*}
	f(x,y)=\sum_{k=1}^r\lambda_k(x+\beta_ky)^d,
\end{equation*}
where $\beta_k\ne\infty$ for $k=1,2,\dots,r$. We refer to this as an $F$-decomposition of $f$. The least $r$ for which the above decomposition holds is called the $F$-rank of $f$, and this is denoted by $\operatorname{FR}\,(f)=r$. An $F$-decomposition of $f$ of length $\operatorname{FR}\,(f)$ is called a minimal $F$-decomposition. Here, the $F$ stands for ``finite", indicating that all the eigenvalues included in the decomposition are finite numbers. 

Given any complex binary form $f$, its minimal $F$-decompositions can be obtained via Theorem \ref{Th 3.1}. If there exists a minimal decomposition of $f$ that does not include the term $y^d$, then a minimal $F$-decomposition of $f$ coincides with a minimal Waring decomposition. If every minimal decomposition of $f$ necessarily includes the term $y^d$, then one can first compute a minimal $F$-decomposition of $f_x=\frac{\partial f}{\partial x}$, and subsequently integrate $f_x$ with respect to $x$ to obtain a minimal decomposition of $f$. As a matter of fact, we have the following observations. 


\begin{proposition}\label{Pro 3.3}
	Let $f\in\mathcal{R}_d$ and denote $f_x=\frac{\partial f}{\partial x}$. Then the following hold: 
	\begin{enumerate}
		\item[(1)] $\operatorname{FR}\,(f)\geqslant\operatorname{WR}\,(f)\geqslant\operatorname{FR}\,(f_x)\geqslant\operatorname{WR}\,(f)-1$. 
		\item[(2)] If $\operatorname{FR}\,(f)=\operatorname{FR}\,(f_x)$, then $\operatorname{WR}\,(f)=\operatorname{FR}\,(f)$. 
		\item[(3)] If $\operatorname{FR}\,(f)>\operatorname{FR}\,(f_x)$, then $\operatorname{WR}\,(f)=\operatorname{FR}\,(f_x)+1$. 
	\end{enumerate}
\end{proposition}

\begin{proof}
	Denote $\operatorname{WR}\,(f)=r$ and $\operatorname{FR}\,(f_x)=s$. 
	
	(1) It is clear that $\operatorname{FR}\,(f)\geqslant\operatorname{WR}\,(f)$. 
	
	We now prove that $\operatorname{WR}\,(f)\geqslant\operatorname{FR}\,(f_x)$. If $f$ has a minimal decomposition that does not include the term $y^d$, let this minimal decomposition of $f$ be given by 
	\begin{equation*}
		f(x,y)=\sum_{k=1}^r\lambda_k(x+\beta_ky)^d. 
	\end{equation*}
	Taking the partial derivative of $f(x,y)$ with respect to $x$, we obtain an $F$-decomposition of $f_x$: 
	\begin{equation*}
		f_x(x,y)=d\sum_{k=1}^r\lambda_k(x+\beta_ky)^{d-1}. 
	\end{equation*}
	Thus, $\operatorname{FR}\,(f_x)\leqslant r=\operatorname{WR}\,(f)$. If there exists a minimal decomposition of $f$ that includes the term $y^d$, similarly we obtain $\operatorname{FR}\,(f_x)\leqslant r-1<\operatorname{WR}\,(f)$. 
	
	Finally, we show that $\operatorname{FR}\,(f_x)\geqslant\operatorname{WR}\,(f)-1$. Assume the contrary, that $\operatorname{FR}\,(f_x)<r-1$. Let one minimal $F$-decomposition of $f_x$ be
	\begin{equation*}
		f_x(x,y)=\sum_{k=1}^s\mu_k(x+\gamma_ky)^{d-1}. 
	\end{equation*}
	Integrating $f_x(x,y)$ with respect to $x$, we obtain a Waring decomposition of $f$: 
	\begin{equation*}
		f(x,y)=\frac{1}{d}\sum_{k=1}^s\mu_k(x+\gamma_ky)^d+\mu_{s+1}y^d. 
	\end{equation*}
	Consequently, $\operatorname{WR}\,(f)\leqslant s+1<r$, which is a contradiction! Therefore, $\operatorname{FR}\,(f_x)\geqslant r-1=\operatorname{WR}\,(f)-1$. 
	
	(2) This follows immediately from (1). 
	
	(3) Suppose that $\operatorname{FR}\,(f)>\operatorname{FR}\,(f_x)$. If $\operatorname{FR}\,(f)=\operatorname{WR}\,(f)$, then it follows that $\operatorname{WR}\,(f) > \operatorname{FR}\,(f_x)$. According to (1), this forces that $\operatorname{FR}\,(f_x)=\operatorname{WR}\,(f)-1$. If $\operatorname{FR}\,(f)>\operatorname{WR}\,(f)$, then every minimal decomposition of  $f$ must include the term $y^d$. Taking the derivative, we see that $\operatorname{WR}\,(f)>\operatorname{FR}\,(f_x)$. Furthermore, in conjunction with (1), it also follows that $\operatorname{FR}\,(f_x)=\operatorname{WR}\,(f)-1$.
\end{proof}

\begin{remark}\label{Rmk 3.4}
	Combining Theorem \ref{Th 3.1} and Proposition \ref{Pro 3.3}, we can solve the Waring problem of any complex binary form. In fact, for any $f\in \mathcal{R}_d$, first we can determine the $F$-ranks of $f$ and $f_x$. If $\operatorname{FR}\,(f)=\operatorname{FR}\,(f_x)$, then $\operatorname{WR}\,(f)=\operatorname{FR}\,(f)$ and so a minimal $F$-decomposition of $f$ is a minimal decomposition. If $\operatorname{FR}\,(f)>\operatorname{FR}\,(f_x)$, then $\operatorname{WR}\,(f)=\operatorname{FR}\,(f_x)+1$. In this case, integrating a minimal $F$-decomposition of $f_x$ with respect to $x$ yields a minimal decomposition of $f$. 
\end{remark}

Furthermore, we have the following results regarding the uniqueness of minimal decompositions of binary forms. Before proceeding, we introduce two notations. Let $q\in \mathbb{Q}$. We define the following: 
\begin{equation*}
	\lfloor q\rfloor=\{p\in\mathbb{Z}\mid q-1<p\leqslant q\},\quad\lceil q\rceil=\{p\in\mathbb{Z}\mid q\leqslant p<q+1\}. 
\end{equation*}
Clearly, when $q\in\mathbb{Z}$, we have $\lfloor q\rfloor=\lceil q\rceil=q$. 

\begin{proposition}\label{Pro 3.5}
	Let $f\in \mathcal{R}_d$ and $r\leqslant \lfloor\frac{d+1}{2}\rfloor$. Suppose there exist $\lambda_1,\lambda_2,\dots,\lambda_r\in\mathbb{C}^*$ and pairwise distinct $\beta_1,\beta_2,\dots,\beta_r\in\mathbb{C}$ such that 
	\begin{equation*}
		f(x,y)=\sum_{k=1}^{r}\lambda_k(x+\beta_ky)^d. 
	\end{equation*}
	Then, $\operatorname{WR}\,(f)=r$, and $f$ has a unique minimal decomposition. 
\end{proposition}

\begin{proof}
	Let $f(x,y)=\sum_{\ell=1}^{s}\mu_\ell(x+\gamma_\ell y)^d$ be a minimal decomposition. Then, $s\leqslant r$. If this minimal decomposition includes the term $y^d$, we place it as the last term, that is, treating $\gamma_s$ as if $\gamma_s=\infty$. In this case, we have 
	\begin{equation*}
		\sum_{k=1}^{r}\lambda_k(x+\beta_ky)^d-\sum_{\ell=1}^{s}\mu_\ell(x+\gamma_\ell y)^d=0.
	\end{equation*}
	Without loss of generality, we may assume that $\beta_1=\gamma_1,\beta_2=\gamma_2,\dots,\beta_t=\gamma_t$, where $0\leqslant t\leqslant s$, and $\{\beta_{t+1},\dots,\beta_r\}\bigcap\newline\{\gamma_{t+1},\dots,\gamma_s\}=\emptyset$. Then, the above equation can be rewritten as
	\begin{equation*}
		\sum_{k=1}^{t}(\lambda_k-\mu_k)(x+\beta_ky)^d+\sum_{k=t+1}^{r}\lambda_k(x+\beta_ky)^d-\sum_{\ell=t+1}^{s}\mu_\ell(x+\gamma_\ell y)^d=0.
	\end{equation*}
	Since $s+r-t\leqslant 2r\leqslant d+1$, by Lemma \ref{lemma 2.1}, the terms $(x+\beta_1y)^d,\dots,(x+\beta_ry)^d,(x+\gamma_{t+1}y)^d,\dots,(x+\gamma_sy)^d$ are linearly independent. This forces that $r=s=t$ and $\lambda_k=\mu_k$ for $k=1,2,\dots,r$. Therefore, $\operatorname{WR}\,(f)=r$, and $f$ has a unique minimal decomposition. 
\end{proof}

We now discuss the Waring decomposition for a generic binary form of degree $d$. Hereafter, we refer to ``a generic binary form satisfying a certain propert'' in the usual sense of algebraic geometry \cite{Cox}, meaning that there exists a polynomial $F$ in the coefficients of $f$ such that all binary forms $f$ for which $F$ is nonzero satisfy this property. In this situation, almost all binary forms satisfy the property except a subset that has measure zero.

We first consider the odd degree case where $d=2r-1$. Employing Theorem \ref{Th 3.1} and its corollary, we demonstrate that the set of all binary forms of degree $2r-1$ of Waring rank $r$ contains a dense open subset of $\mathcal{R}_{2r-1}$, which provides an elementary proof of the famous Sylvester's 1851 Theorem \cite{Sylvester}.

\begin{corollary}\label{Cor 3.6}
	Let $d=2r-1$. For a generic binary form $f(x,y)=\sum_{i=0}^{d}\binom{d}{i}a_ix^{d-i}y^i\in \mathcal{R}_d$, there exist $\lambda_1,\lambda_2,\dots,\lambda_r\in\mathbb{C}^*$ and pairwise distinct $\beta_1,\beta_2,\dots,\beta_r\in\mathbb{C}$ such that 
	\begin{equation*}
		f(x,y)=\sum_{k=1}^{r}\lambda_k(x+\beta_ky)^d. 
	\end{equation*}
	In particular, the Waring rank of $f$ is $r$, and the minimal decomposition is unique. 
\end{corollary}

\begin{proof}
	We retain the notations introduced in Theorem \ref{Th 3.1}. Denote the determinant of $A_0$ by $\operatorname{det}\,(A_0)=D(a_0,a_1,\dots,\newline a_{d-1})$. Let $T_d(\lambda)=\operatorname{det}\,(\lambda A_0-A_1)$ and $T_d'(\lambda)=\frac{\operatorname{d}}{\operatorname{d}\lambda}T_d(\lambda)$, and we denote the resultant of $T_d$ and $T_d'$ with respect to $\lambda$ by $\operatorname{Res}\,(T_d,T_d',\lambda)=R(a_0,a_1,\dots,a_d)$. It is well known that $T_d(\lambda)$ has no repeated roots if and only if $\operatorname{Res}\,(T_d,T_d',\lambda)\ne 0$. 
	
	According to Theorem \ref{Th 3.1}, there exist $\lambda_1,\lambda_2,\dots,\lambda_r\in\mathbb{C}^*$ and pairwise distinct $\beta_1,\beta_2,\dots,\beta_r\in\mathbb{C}$ such that 
	\begin{equation*}
		f(x,y)=\sum_{k=1}^{r}\lambda_k(x+\beta_ky)^d,
	\end{equation*}
	if and only if conditions (a), (b), and (c) in item (2) of Theorem \ref{Th 3.1} are all satisfied. 
	
	If (b) holds, that is, $\operatorname{rank}\,(A_0)=r$, then the system of equations 
	\begin{equation*}
		\begin{pmatrix}
			a_0&a_1&\cdots&a_{r-1}\\
			a_1&a_2&\cdots&a_r\\
			\vdots&\vdots&\ddots&\vdots\\
			a_{r-1}&a_r&\cdots&a_{2r-2}
		\end{pmatrix}
		\begin{pmatrix}
			x_0\\
			x_1\\
			\vdots\\
			x_{r-1}
		\end{pmatrix}=
		\begin{pmatrix}
			a_r\\
			a_{r+1}\\
			\vdots\\
			a_{2r-1}
		\end{pmatrix}
	\end{equation*}
	has a solution, and thus (a) holds. Therefore, items (a), (b), and (c) are simultaneously satisfied if and only if so are\newline (b) and (c). Specifically, $\operatorname{rank}\,(A_0)=r$ and $T_d(\lambda)$ has no repeated roots, which is equivalent to $D(a_0,a_1,\dots,a_{d-1})R(a_0,\newline a_1,\dots,a_d)\ne 0$. 
	
	By Proposition \ref{Pro 3.5}, the Waring rank of $\sum_{k=1}^{r}(x+ky)^{2r-1}$ is $r$. Hence, $DR$ is a nonzero polynomial. Consequently, the set of points where $DR\ne 0$ forms a dense open set of $\mathbb{A}^{d+1}$, see \cite{Cox}. 
\end{proof}

We now turn to the Waring decompositions of even degree generic binary forms. The relevant results were first established by Gundelfinger in 1887 by extending the Hessian to higher orders \cite{Gundelfinger}. Herein, we develop a natural connection between the Waring decompositions of binary forms of odd and even degrees via integration and differentiation, and provide an elementary proof of Gundelfinger's Theorem with a help of the results for odd forms. 


\begin{corollary}\label{Cor 3.7}
	Let $d=2r-2$. For a generic binary form $f(x,y)=\sum_{i=0}^{d}\binom{d}{i}a_ix^{d-i}y^i\in \mathcal{R}_d$, there exist $\lambda_1,\lambda_2\dots,\lambda_r\in\mathbb{C}^*$ and pairwise distinct $\beta_1,\beta_2,\dots,\beta_r\in\mathbb{C}$ such that 
	\begin{equation*}
		f(x,y)=\sum_{k=1}^{r}\lambda_k(x+\beta_ky)^d, 
	\end{equation*}
	and the Waring rank of $f$ is $r$. Moreover, a generic binary form of degree $d$ and Waring rank $r$ admits infinitely many such decompositions. 
\end{corollary}

\begin{proof}		
	Taking the indefinite integral of $f(x,y)$ with respect to $x$, we obtain 
	\begin{equation*}
		\tilde{f}(x,y)=\int f(x,y)\,\mathrm{d}x=\sum_{i=0}^{d}\binom{d}{i}\frac{a_i}{d+1-i}x^{d+1-i}y^i+\frac{a_{d+1}}{d+1}y^{d+1}=\frac{1}{d+1}\sum_{i=0}^{d+1}\binom{d+1}{i}a_ix^{d+1-i}y^i. 
	\end{equation*}
	We retain the notations $D(a_0,a_1,\dots,a_d)$ and $R(a_0,a_1,\dots,a_{d+1})$ from the proof of Corollary \ref{Cor 3.6}. Then $DR$ is a nonzero polynomial. Let the highest degree of $a_{d+1}$ in $DR$ be $t$, and we can express $DR$ in the following form: 
	\begin{equation*}
		D(a_0,a_1,\dots,a_d)R(a_0,a_1,\dots,a_{d+1})=\sum_{j=0}^{t}B_j(a_0,a_1,\dots,a_d)a_{d+1}^j. 
	\end{equation*}
	It is easy to see that $B_t(a_0,a_1,\dots,a_d)$ is a nonzero polynomial. Consequently, the set of points where $B_t\ne 0$ forms a dense open subset of $\mathbb{A}^{d+1}$, which we denote by $V$. 
	
	Take $(a_0,a_1,\dots,a_d)\in V$. Then $DR$ is a nonzero polynomial in $a_{d+1}$, and thus has only finitely many zeros. Choosing $a_{d+1}$ such that $DR\ne 0$, it follows from Corollary \ref{Cor 3.6} that the Waring rank of $\tilde{f}(x,y)$ is $r$. Moreover, there exist $\lambda_1,\lambda_2,\dots,\lambda_r\in\mathbb{C}^*$ and pairwise distinct $\beta_1,\beta_2,\dots,\beta_r\in\mathbb{C}$ such that 
	\begin{equation*}
		\tilde{f}(x,y)=\frac{1}{d+1}\sum_{k=1}^{r}\lambda_k(x+\beta_ky)^{d+1}. 
	\end{equation*}
	Now taking the partial derivative of $\tilde{f}(x,y)$ with respect to $x$, we obtain a Waring decomposition of $f(x,y)$ as follows: 
	\begin{equation}\label{3.6}
		f(x,y)=\frac{\partial}{\partial x}\tilde{f}(x,y)=\frac{1}{d+1}\sum_{k=1}^{r}\lambda_k(d+1)(x+\beta_ky)^d=\sum_{k=1}^{r}\lambda_k(x+\beta_ky)^d.
	\end{equation}
	For any $s<r$, we have $s+r\leqslant 2r-1=d+1$. Proceeding as in the proof of Proposition \ref{Pro 3.5}, we can show that $f$ does not admit a Waring decomposition of length $s$. Consequently, the Waring rank of $f$ is precisely $r$. 
	
	Finally, we demonstrate that a generic binary form $f(x,y)$ of degree $d$ and Waring rank $r$ admits infinitely many Waring decompositions of the form given in \eqref{3.6}. 
	
	Still let $f(x,y)=\sum_{i=0}^{d}\binom{d}{i}a_ix^{d-i}y^i$ and let $\tilde{f}(x,y)$ be as stated above. We denote by $W$ the set of points for which $D(a_0,a_1,\dots,a_{d-2})B_t(a_0,a_1,\dots,a_d)\ne 0$. It follows that $W$ is a dense open subset of $\mathbb{A}^{d+1}$. Take $(a_0,a_1,\dots,a_d)\in W$. Then, for $\tilde{f}(x,y)$, we have 
	\begin{align*}
		T_d(\lambda)&=\operatorname{det}\,(\lambda A_0-A_1)\\
		&=
		\begin{vmatrix}
			\lambda a_0-a_1&\lambda a_1-a_2&\cdots&\lambda a_{r-1}-a_r\\
			\lambda a_1-a_2&\lambda a_2-a_3&\cdots&\lambda a_r-a_{r+1}\\
			\vdots&\vdots&\ddots&\vdots\\
			\lambda a_{r-1}-a_r&\lambda a_r-a_{r+1}&\cdots&\lambda a_d-a_{d+1}
		\end{vmatrix}\\
		&=-a_{d+1}\begin{vmatrix}
			\lambda a_0-a_1&\lambda a_1-a_2&\cdots&\lambda a_{r-2}-a_{r-1}\\
			\lambda a_1-a_2&\lambda a_2-a_3&\cdots&\lambda a_{r-1}-a_r\\
			\vdots&\vdots&\ddots&\vdots\\
			\lambda a_{r-2}-a_{r-1}&\lambda a_{r-1}-a_r&\cdots&\lambda a_{d-2}-a_{d-1}
		\end{vmatrix}+S(\lambda)\\
		&=-a_{d+1}T_{d-2}(\lambda)+S(\lambda),
	\end{align*}
	where $S(\lambda)$ is the polynomial comprising the terms in the determinant expansion that do not involve $a_{d+1}$. Given that $D(a_0,a_1,\dots,a_{d-2})\ne 0$, the polynomial $T_{d-2}(\lambda)$ is nonzero. Hence, for distinct values of $a_{d+1}$, the corresponding $T_d(\lambda)$'s are pairwise distinct. On the other hand, there exist infinitely many distinct $a_{d+1}$ such that $DR\ne 0$, and the corresponding $T_d(\lambda)$'s have no repeated roots. Consequently, these $T_d(\lambda)$'s have distinct sets of zeros. According to Theorem \ref{Th 3.1}, the $\tilde{f}(x,y)$'s corresponding to different $a_{d+1}$ have different Waring decompositions. Then, it follows from \eqref{3.6} that $f(x,y)$ admits infinitely many minimal decompositions. 
\end{proof}

Finally, we present an algorithm for computing the Waring rank and a minimal decomposition of any binary form, as detailed in Remark \ref{Rmk 3.4}. To elucidate the underlying idea of the algorithm, we begin with a concrete example. 


\begin{example} 
	Consider the binary form $f(x,y)=8x^3+12x^2y+6xy^2$. We first determine the $F$-ranks of $f$ and $f_x$, and then determine the Waring rank of $f$ in accordance with Proposition \ref{Pro 3.3}. Thereafter, we perform a minimal decomposition based on the specific context. 
	
	Initially, determine the $F$-rank of $f$ by the method outlined in Theorem \ref{Th 3.1}. For the given $f$, we have $a_0=8,a_1=4,a_2=2$ and $a_3=0$. First, solve for $c_0,c_1,\dots,c_{r-1}$ such that $a_i=-(c_0a_{i-r}+c_1a_{i-r+1}+\cdots+c_{r-1}a_{i-1})$ for $i=r,r+1,\dots,d$. 
	
	For $r=1$, the system of equations
	\begin{equation*}
		\begin{pmatrix}
			8\\
			4\\
			2
		\end{pmatrix}
		\begin{pmatrix}
			c_0
		\end{pmatrix}=
		\begin{pmatrix}
			-4\\
			-2\\
			0
		\end{pmatrix}
	\end{equation*}
	has no solution. 
	
	For $r=2$, the system of equations
	\begin{equation*}
		\begin{pmatrix}
			8&4\\
			4&2
		\end{pmatrix}
		\begin{pmatrix}
			c_0\\
			c_1
		\end{pmatrix}=
		\begin{pmatrix}
			-2\\
			0
		\end{pmatrix}
	\end{equation*}
	also has no solution. 
	
	For $r=3$, solving the equation 
	\begin{equation*}
		\begin{pmatrix}
			8&4&2
		\end{pmatrix}
		\begin{pmatrix}
			c_0\\
			c_1\\
			c_2
		\end{pmatrix}=
		\begin{pmatrix}
			0
		\end{pmatrix},
	\end{equation*}
	we obtain a parametric solution $(c_0,c_1,c_2)=(a,b,-4a-2b)$, where $a,b\in\mathbb{C}$. 
	
	Let $a_4=-(4a+2b+0)=-4a-2b$. Then
	\begin{equation*}
		A_0=\begin{pmatrix}
			8&4&2\\
			4&2&0\\
			2&0&-4a-2b
		\end{pmatrix}. 
	\end{equation*}
	At this point, $\operatorname{det}\,(A_0)=-8\ne0$, and $T(\lambda)=\lambda^3-(4a+2b)\lambda^2+b\lambda+a$. In fact, it suffices to take $a=0$ and $b=-\frac{1}{8}$. Then $T(\lambda)$ has no repeated roots, and thus $\operatorname{FR}\,(f)=3$. 
	
	The next thing to do is to determine the $F$-rank of $f_x$. Notice that $f_x=24x^2+24xy+6y^2=6(2x+y)^2$. Hence, $\operatorname{FR}\,(f_x)=1<\operatorname{FR}\,(f)$. 
	
	Using Proposition \ref{Pro 3.3}, we can now derive that $\operatorname{WR}\,(f)=\operatorname{FR}\,(f_x)+1=2<\operatorname{FR}\,(f)$. Therefore, each minimal decomposition of $f$ must include the term $y^3$. By integrating $f_x$ with respect to $x$, we obtain a minimal decomposition of $f$: 
	\begin{equation*}
		f(x,y)=(2x+y)^3-y^3. 
	\end{equation*}
	Moreover, by Proposition \ref{Pro 3.5}, the above decomposition is the unique minimal decomposition of $f$.
\end{example}

The following presents an algorithm for determining the Waring rank of any given binary form and its minimal decompositions. 

\begin{algorithm}\label{Alg 3.9}
	Given $f(x,y)=\sum_{i=0}^d\binom{d}{i}a_ix^{d-i}y^i\in\mathcal{R}_d$. 
	
	\begin{enumerate}
		\item[] \textit{Step 1: Solve the system of linear equations.} For $r=1$, solve the following system of equations: 
		\begin{equation}\label{3.7}
			\begin{pmatrix}
				a_0&a_1&\cdots&a_{r-1}\\
				a_1&a_2&\cdots&a_r\\
				\vdots&\vdots&\ddots&\vdots\\
				a_{d-r}&a_{d-r+1}&\cdots&a_{d-1}
			\end{pmatrix}
			\begin{pmatrix}
				x_0\\
				x_1\\
				\vdots\\
				x_{r-1}
			\end{pmatrix}=
			\begin{pmatrix}
				-a_r\\
				-a_{r+1}\\
				\vdots\\
				-a_d
			\end{pmatrix}. 
		\end{equation}
		If this system has a solution, select a set of parametric solutions $(c_0,c_1,\dots,c_{r-1})$ and continue; otherwise, repeat Step 1 for $r+1$. 
		
		\item[] \textit{Step 2: Compute the resultant to determine the $F$-rank of $f$.} Let $T(\lambda)=\lambda^r+c_{r-1}\lambda^{r-1}+\cdots+c_0$ and $T'(\lambda)=\frac{\operatorname{d}}{\operatorname{d}\lambda}T(\lambda)$. Compute the resultant of $T$ and $T'$ with respect to $\lambda$, denoted by $\operatorname{Res}\,(T,T',\lambda)$, which is a polynomial in $c_0,c_1,\dots,c_{r-1}$. If there exist $c_0,c_1,\dots,c_{r-1}\in\mathbb{C}$ such that $\operatorname{Res}\,(T,T',\lambda)\ne 0$, then we have $\operatorname{FR}\,(f)=r$ and continue; otherwise, return to Step 1 for $r+1$. 
		
		
		
		\item[] \textit{Step 3: Determine the $F$-rank of $f_x$.} Apply the above steps to $f_x$ to obtain $\operatorname{FR}\,(f_x)=s$. 
		
		\item[] \textit{Step 4: Determine the Waring rank of $f$.} If $s=r$, then $\operatorname{WR}\,(f)=r$. If $s<r$, then $\operatorname{WR}\,(f)=s+1$. 
		
		\item[] \textit{Step 5: Construct a minimal decomposition of $f$.} If $\operatorname{WR}\,(f)=\operatorname{FR}\,(f)$, then take an appropriate set $(c_0,c_1,\dots,c_{r-1})$ from Step 2, and obtain the polynomial $T(\lambda)=\lambda^r+c_{r-1}\lambda^{r-1}+\cdots+c_0$, whose distinct roots are $\beta_1,\beta_2,\dots,\beta_r$. Then, a minimal decomposition of $f$ is given by
		\begin{equation*}
			f(x,y)=\sum_{k=1}^{r}\lambda_k(x+\beta_ky)^d.
		\end{equation*}
		The coefficients $\lambda_1,\lambda_2,\dots,\lambda_r$ can be determined using the method of undetermined coefficients. 
		
		If $\operatorname{WR}\,(f)=\operatorname{FR}\,(f_x)+1$, then first construct a minimal $F$-decomposition $f_x(x,y)=\sum_{k=1}^s\mu_k(x+\gamma_ky)^{d-1}$. Then, integrate $f_x$ with respect to $x$ to obtain a minimal decomposition of $f$: 
		\begin{equation*}
			f(x,y)=\frac{1}{d}\sum_{k=1}^s\mu_k(x+\gamma_ky)^d+\mu_{s+1}y^d, 
		\end{equation*}
		where $\mu_{s+1}=a_d-\frac{1}{d}\sum_{k=1}^s\mu_k\gamma_k^d$.
	\end{enumerate}
\end{algorithm}

\begin{remark}
	Unlike Theorem \ref{Th 3.1}, we do not need to determine whether $\operatorname{det}\,(A_0)\ne 0$ in the step of the above algorithm for determining the $F$-rank of $f$. In fact, equation \eqref{3.1} in Theorem \ref{Th 3.1} is equivalent to \[(\partial_y^r+c_{r-1}\partial_y^{r-1}\partial_x+\cdots+c_0\partial_x^r)\circ f=0. \]
	If $T(\lambda)$ has pairwise distinct roots $\beta_1,\beta_2,\dots,\beta_r$, then \[\partial_y^r+c_{r-1}\partial_y^{r-1}\partial_x+\cdots+c_0\partial_x^r=\prod_{j=1}^r(\partial_y-\beta_j\partial_x), \]
	hence by the Apolarity Lemma (Theorem \ref{Th 2.4}), $f$ can be expressed in the following form:
	\begin{equation*}
		f(x,y)=\sum_{k=1}^{r}\lambda_k(x+\beta_ky)^d. 
	\end{equation*}
	Therefore, for a given $f\in\mathcal{R}_d$, in the process of determining the $F$-rank of $f$ according to Algorithm \ref{Alg 3.9}, if $r<\operatorname{FR}\,(f)$, then either the system of equations \eqref{3.7} has no solution, or $T(\lambda)$ has repeated roots. In either case, we need to return to Step 1 with $r+1$. 
	Particularly, if $r=\operatorname{FR}\,(f)$, then $\operatorname{det}\,(A_0)\ne 0$ holds automatically according to Theorem \ref{Th 3.1}. 
\end{remark}

To conclude, we discuss a special example. 

\begin{example}
	Consider the form $f(x,y)=3x^2y$, where $a_1=1$ and $a_0=a_2=a_3=0$. First, we determine the $F$-rank of $f$. According to Step 1 of the algorithm, solve the system of linear equations.  
	
	For $r=1$, the system of equations 
	\begin{equation*}
		\begin{pmatrix}
			0\\
			1\\
			0
		\end{pmatrix}
		\begin{pmatrix}
			c_0
		\end{pmatrix}=
		\begin{pmatrix}
			-1\\
			0\\
			0
		\end{pmatrix}
	\end{equation*}
	has no solution. 
	
	Let $r=2$. Consider the following equations: 
	\begin{equation*}
		\begin{pmatrix}
			0&1\\
			1&0
		\end{pmatrix}
		\begin{pmatrix}
			c_0\\
			c_1
		\end{pmatrix}=
		\begin{pmatrix}
			0\\
			0
		\end{pmatrix}. 
	\end{equation*}
	We find that the unique solution is $c_0=c_1=0$. However, in this case, the characteristic polynomial $T(\lambda)=\lambda^2$ has a repeated root. 
	
	Consider $r=3$. Solve the equation 
	\begin{equation*}
		\begin{pmatrix}
			0&1&0
		\end{pmatrix}
		\begin{pmatrix}
			c_0\\
			c_1\\
			c_2
		\end{pmatrix}=
		\begin{pmatrix}
			0
		\end{pmatrix}, 
	\end{equation*}
	and we obtain a parametric solution $(c_0,c_1,c_2)=(a,0,b)$, where $a,b\in\mathbb{C}$. 
	
	At this point, we have $T(\lambda)=\lambda^3+b\lambda^2+a$. By setting $a=-1$ and $b=0$, we find that $T(\lambda)$ has no repeated root. Therefore, $\operatorname{FR}\,(f)=3$. 
	
	Next, we determine the $F$-rank of $f_x(x,y)=6xy$. It is clear that $f_x$ is not a perfect square. Moreover, we have
	\begin{equation}\label{3.8}
		f_x(x,y)=\frac{3}{2}(x+y)^2-\frac{3}{2}(x-y)^2,
	\end{equation}
	so $\operatorname{FR}\,(f_x)=2<\operatorname{FR}\,(f)$. According to Step 4, we obtain $\operatorname{WR}\,(f)=2+1=3$. 
	
	Finally, we construct a minimal decomposition of $f$. Since $\operatorname{WR}\,(f)=\operatorname{FR}\,(f)$, by Step 5, a minimal $F$-decomposition of $f$ coincides with a minimal decomposition. We find that the roots of $T(\lambda)=\lambda^3-1$ are $1,\omega,\omega^2$, where $\omega=e^{\frac{2\pi\operatorname{i}}{3}}$. Hence, along with the method of undetermined coefficients we obtain a minimal decomposition of $f$: 
	\begin{equation*}
		f(x,y)=\frac{1}{3}(x+y)^3+\frac{\omega^2}{3}(x+\omega y)^3+\frac{\omega}{3}(x+\omega^2 y)^3. 
	\end{equation*}
	
	On the other hand, note that $\operatorname{WR}\,(f)=\operatorname{FR}\,(f_x)+1$. Likewise, according to Step 5, we first present a minimal $F$-decomposition of $f_x$ as shown in \eqref{3.8}. Subsequently, integrating $f_x$ with respect to $x$ yields a minimal decomposition of $f$ that includes the term $y^3$:
	\begin{equation*}
		f(x,y)=\frac{1}{2}(x+y)^3-\frac{1}{2}(x-y)^3-y^3. 
	\end{equation*}
\end{example}

\section*{Use of AI tools declaration}
The authors declare they have not used Artificial Intelligence (AI) tools in the creation of this article.

\section*{Acknowledgements}
H.-L. Huang was supported by Key Program of Natural Science Foundation of Fujian Province (Grant no. 2024J02018) and National Natural Science Foundation of China (Grant no. 12371037). Y. Ye was supported by the National Key R\&D Program of China (No. 2024YFA1013802), the National Natural Science Foundation of China (Nos. 12131015 and 12371042), and the Quantum Science and Technology-National Science and Technology Major Project (No. 2021ZD0302902).

\section*{Conflict of interest}
The authors declare no conflicts of interest.

\begin{spacing}{.88}
	\setlength{\bibsep}{2.pt}

\end{spacing}
\end{document}